\documentclass[11pt,twoside]{article}
\usepackage{acta-m,euler}

\title{%
Generalized GCD matrices
}

\newcommand{\vol}{\textbf{2}, 2 (2010) 160--167}   %% to be completed by the editor, do not modify it
\newcommand{\amss}{\amssubj{11C20, 11A25, 15A36
}}   %% Source:   http://www.ams.org/msc/
\newcommand{\keyw}{\keywords{GCD matrix, Smith determinant, arithmetical function
}}
\begin{document}
\maketitle

%% SINGLE AUTHOR. If you are a single author, please, use the following command and delete the \twoauthors command completely.

\oneauthor{Antal Bege
 %% Write here the name of the author
}{
Sapientia--Hungarian University of Transylvania\\
Department of Mathematics and Informatics,\\
T\^argu Mure\c{s}, Romania
 %% Write here your affiliation including maybe your address. You can use \\ for line breaks.
}{abege@ms.sapientia.ro
 %% Write here your email.
}

%% TWO AUTHORS. If there are two authors, please, use the following command and and delete the \oneauthor command completely.

%\twoauthors{Author1
%% Write here the first author's name.
%}{Affiliation1
 %% Write here the first author's affiliation including maybe his address. You can use \\ for line breaks.
%}{email1@com
 %% Write here the first author's email.
%}{Author2
 %% Write here the second author's name.
%}{Affiliation2
 %% Write here the second author's affiliation including maybe his address. You can use \\ for line breaks.
%}{email2@com
 %% Write here the first author's email.
%}

%% MORE AUTHORS. For more than two authors, please, use both commands, maybe more than once.

%% Short name of the authors and short title, to be included in heading.

\short{A. Bege}{Generalized GCD matrices}

\begin{abstract}
Let $f$ be an arithmetical function. The matrix $[f(i,j)]_{n\times n}$ given by the value of $f$ in greatest common divisor of $(i,j)$, $f\big((i,j)\big)$ as its $i,\; j$ entry is called the greatest common divisor (GCD) matrix. We consider the generalization of this matrix where the elements are in the form $f\big(i,(i,j)\big)$.
%% Write here the abstract of your paper.
\end{abstract}

%% Write here your paper using the section command, theorem-like environments as:
%% definition, theorem, lemma, corollary, criterion, example, exercise, notation, problem, proposition, remark.
%% For proof you can use the proof environment.

%% For references use the following format. For citation use e. g. \cite{lothaire}.

\section{Introduction}
The classical Smith determinant was introduced in 1875 by H. J. S. Smith \cite{smith1} who also proved that
\begin{equation}
\label{1}
\det [(i,j)]_{n\times n}=
\left|
\begin{array}{cccc}
(1,1)&(1,2)&\cdots&(1,n)\\
(2,1)&(2,2)&\cdots&(2,n)\\
\cdots&\cdots&\cdots&\cdots\\
(n,1)&(n,2)&\cdots&(n,n)\\
\end{array}
\right|
=\varphi (1) \varphi(2)\cdots \varphi(n),
\end{equation}
where $(i,j)$ represents the greatest common divisor of $i$ and $j$, and $\varphi(n)$ denotes the  Euler's totient function.
\\
The GCD matrix with respect to $f$ is
\[
[f(i,j)]_{n\times n}=
\left[
\begin{array}{cccc}
f((1,1))&f((1,2))&\cdots&f((1,n))\\
f((2,1))&f((2,2))&\cdots&f((2,n))\\
\cdots&\cdots&\cdots&\cdots\\
f((n,1))&f((n,2))&\cdots&f((n,n))\\
\end{array}
\right].
\]
If we consider the GCD matrix $[f(i,j)]_{n\times n}$, where
\[
f(n)=\sum_{d\mid n}g(d),
\]
H. J. Smith proved that
\[
\det[f(i,j)]_{n\times n}=g(1)\cdot g(2)\cdots g(n).
\]
For $g=\varphi $
\[
f(i,j)=\sum_{d\mid(i,j)}\varphi(d)=(i,j),
\]
this formula reduces to (\ref{1}). Many generalizations of Smith determinants have been presented in literature, see
\cite{beslin1,  gyires1, hauk1,  carthy1, bhat1}.
\\
If we consider the GCD matrix $[f(i,j)]_{n\times n}$ where $f(n)=\sum_{d\mid n}g(d)$ P\'olya and Szeg\H{o} \cite {polya1} proved that
\begin{equation}
\label{2}
[f(i,j)]_{n\times n}=G \cdot C^T,
\end{equation}
where $G$ and $C$ are lower triangular matrices given by
\[
g_{ij}=
\left\{
\begin{array}{cc}
g(j),&j\mid i\\
0,&\mbox{ otherwise }
\end{array}
\right.
\]
and
\[
c_{ij}=
\left\{
\begin{array}{cc}
1,&j\mid i\\
0,&\mbox{ otherwise }
\end{array}
\right..
\]
L. Carlitz \cite{carlitz1} in 1960 gave a new form of (\ref{2})
\begin{equation}
\label{3}%1_7}
[f(i,j)]_{n\times n}=C_n \mbox{ diag}\big(g(1),\; g(2),\ldots ,g(n)\big)\;  C_n^T,
\end{equation}
where
$C_n=[c_{ij}]_{n\times n}$,
\[
c_{ij}=
\left\{
\begin{array}{cc}
1,&j\mid i\\
0,&j \not| \;  i
\end{array}
\right. ,
\]
$D=[d_{ij}]_{n\times n}$ diagonal matrix
\[
d_{ij}=
\left\{
\begin{array}{cc}
g(i),&i=j\\
0,&i \not= j
\end{array}
\right. .
\]
From (\ref{3}) it follows that the value of the determinant is
\begin{equation}
\label{3_1}%1_7}
\det [f(i,j)]_{n\times n}=g(1)g(2)\cdots g(n).
\end{equation}
There are quite a few generalized forms of GCD matrices, which can be found in several references 
\cite{beslin2, bourque1,  hauk2, hong1, hong2, ovall1}.
\\
In this paper we study matrices which have as variables the gratest common divisor and the indices:
\[
[f(i,j)]_{n\times n}=
\left[
\begin{array}{cccc}
f(1,(1,1))&f(1,(1,2))&\cdots&f(1,(1,n))\\
f(2,(2,1))&f(2(2,2))&\cdots&f(2,(2,n))\\
\cdots&\cdots&\cdots&\cdots\\
f(n,(n,1))&f(n,(n,2))&\cdots&f(n,(n,n))\\
\end{array}
\right].
\]
\section{Generalized GCD matrices}

\begin{theorem}
\label{1_1}
%{\bf Theorem 2 [A.B. 2008]}
For a given arithmetical function $g$ let
\[
f(n,m)=\sum_{d\mid n}g(d)-\sum_{d\mid (n,m)}g(d).
\]
Then
\[
[f(i,j)]_{n\times n}=C_n \emph{\mbox{ diag}}[g(1),g(2),\ldots ,g(n)]\, D_n^T,
\]
where
$C_n=[c_{ij}]_{n\times n}$,
\[
c_{ij}=
\left\{
\begin{array}{cc}
1,&j\mid i\\
0,&j \not| \;  i
\end{array}
\right. ,
\]
$D_n=[d_{ij}]_{n\times n}$,
\[
d_{ij}=
\left\{
\begin{array}{cc}
1,&j\not|\; i\\
0,&j \mid  i
\end{array}
\right. .
\]
\end{theorem}
\begin{proof}
After multiplication, the general element of
\[
A=[a_{ij}]_{n\times n}=C\mbox{ diag}[g(1),g(2),\ldots ,g(n)]D^T
\]
is
\[
a_{ij}=\sum_{
\begin{array}{c}
k\mid i\\
k\not|\; j
\end{array}}
g(k)=\sum_{d\mid n}g(d)-\sum_{d\mid (n,m)}g(d)=f(i,j).
\]
\end{proof}

\noindent
{\bf Particular cases}
\\
{\bf 1.} If $g(n)=\varphi(n)$ then
\[
f(n,m)=\sum_{d\mid n}\varphi(d)-\sum_{d\mid (n,m)}\varphi(d)=n-(n,m).
\]
We have the following decomposition:
\[
[i-(i,j)]_{n\times n}=
\left[
\begin{array}{cccc}
1-(1,1)&1-(1,2)&\cdots&1-(1,n)\\
2-(2,1)&2-(2,2)&\cdots&2-(2,n)\\
\cdots&\cdots&\cdots&\cdots\\
n-(n,1)&n-(n,2)&\cdots&n-(n,n)\\
\end{array}
\right].
\]
{\bf 2.} If $g(n)=1$ then
\[
f(n,m)=\tau(n)-\tau(n,m)
\]
and
\[
[\tau(i)-\tau(i,j)]_{n\times n}=C_n \mbox{ diag}\big(1,\; 1,\ldots ,1\big)\;  D_n^T.
\]
{\bf 3.} Let $g(n)=\mu(n)$. From
\[
f(n,m)=\sum_{d\mid n}\mu(d)-\sum_{d\mid (n,m)}\mu(d)=
\left\{
\begin{array}{cl}
0,& n=1\\
0,&n>1,m>1, (n,m)>1\\
-1,&\mbox{othewise }
\end{array}
\right. .
\]
we have
\[
[f(i,j)]_{n\times n}=C_n \mbox{ diag}\big(\mu(1),\; \mu(2),\ldots ,\mu(n)\big)\;  D_n^T.
\]
{\bf 4.} For $g(n)=n$, $f(n,m)=\sigma(n)-\sigma((n,m))$ and
\[
[f(i,j)]_{n\times n}=C_n \mbox{ diag}\big(1,\; 2,\ldots ,n\big)\;  D_n^T.
\]

\noindent
{\bf Remarks}
\\
{\bf 1.}
Due to the fact that the first line of the matrix $[f(i,j)]_{n\times n}$ contains only 0-s, the determinant of the matrix will always be 0.
\\
\\
{\bf 2.}
We can determine the value of the matrix associated with $f$, if the function $f$ is of the form
\[
f(n,m)=h(n)-h\big((n,m)\big).
\]
By using the M\"obius inversion formula, we get
\[
g(n)=\sum_{d\mid n}\mu(d)h\left(\frac{n}{d}\right),
\]
consequently by using Theorem \ref{1_1}, the matrix can be decomposed according to the function $h(n)$:
\[
[f(i,j)]_{n\times n}=C_n\mbox{ diag}[(\mu *h)(1),(\mu *h)(2),\ldots ,(\mu *h)(n)]\, D_n^T.
\]

\begin{theorem}
\label{1_2}
For a given arithmetical function $g$ let
\[
f(i,j)=\sum_{k=1}^ng(k)-\sum_{d\mid i}g(d)-\sum_{d\mid j}g(d)+\sum_{d\mid (i,j)}g(d).
\]
Then
\[
[f(i,j)]_{n\times n}=D_n\mbox{ diag}[g(1),g(2),\ldots ,g(n)]\, D_n^T,
\]
where
$D_n=[d_{ij}]_{n\times n}$,
\[
d_{ij}=
\left\{
\begin{array}{cc}
1,&j\not|\; i\\
0,&j \mid  i
\end{array}
\right. .
\]
\end{theorem}
\begin{proof}
After multiplication, the general element of the matrix
\[
A=[a_{ij}]_{n\times n}=D_n\mbox{ diag}[g(1),g(2),\ldots ,g(n)]\, D_n^T
\]
is
\begin{eqnarray*}
a_{ij}&=&
\sum_{
\begin{array}{c}
k\not|\; n\\
k\not|\; m
\end{array}}g(k)=\sum_{k=1}^ng(k)-
\sum_{
k\mid n  \mbox{ or }
k\mid  m
}g(k)=\\
&=&\sum_{k=1}^ng(k)-\sum_{k\mid n}g(k)-\sum_{k\mid m}g(k)+\sum_{k\mid (n,m)}g(k)=f(i,j).
\end{eqnarray*}
\end{proof}

\noindent
{\bf Particular cases}
\\
{\bf 1.} If $g(n)=\varphi(n)$ then
\[
f(i,j)=\sum_{k=1}^n\varphi(k)-i-j+(i,j),
\]
\[
[f(i,j)]_{n\times n}=D_n\mbox{ diag}[\varphi(1),\varphi(2),\ldots ,\varphi(n)]\, D_n^T.
\]
{\bf 2.} If $g(n)=1$ then
\[
f(i,j)=n-\tau(i)-\tau(j)+\tau(i,j)
\]
and
\[
[f(i,j)]_{n\times n}=D_n \mbox{ diag}\big(1,\; 1,\ldots ,1\big)\;  D_n^T.
\]
{\bf 3.} $g(n)=n$. Then
\[
f(i,j)=\frac{n(n+1)}{2}-\sigma(n)-\sigma(m)+\sigma((n,m))
\]
and
\[
[f(i,j)]_{n\times n}=D_n \mbox{ diag}\big(1,\; 2,\ldots ,n\big)\;  D_n^T.
\]
Another generalization is the following:

\begin{theorem}
For a given arithmetical function $g$ let
\[
f(i,j)=\sum_{k=1}^ng(k)-\sum_{d\mid i}g(d)-\sum_{d\mid j}g(d)+\sum_{d\mid (i,j)}g(d).
\]
We define the following $A=[a_{ij}]_{n\times n}$ matrix
\[
a_{ij}=
\left\{
\begin{array}{cc}
f(i,j),&i,j>1\\
g(1)+f(i,j),&i=1 \mbox{ or }j=1
\end{array}
\right. .
\]
Then
\[
A=D_n'\emph{\mbox{ diag}}[g(1),g(2),\ldots ,g(n)]\, D_n'^T,
\]
where
$D_n'=[d'_{ij}]_{n\times n}$,
\[
d'_{ij}=
\left\{
\begin{array}{cc}
1,&i=j=1\\
d_{ij},&ij\not=1
\end{array}
\right. .
\]
%and
%\[
%\det A=g(1)g(2)\cdots g(n).
%\]
\end{theorem}
\begin{proof}
\\
We calculate the general element of the matrix
\[
B=[a_{ij}]_{n\times n}=D_n'{\mbox{ diag}}[g(1),g(2),\ldots ,g(n)]\, D_n'^T.
\]
If $i>1$ or $j>1$ we have
\begin{eqnarray*}
b_{ij}&=&
\sum_{
\begin{array}{c}
k\not|\; n\\
k\not|\; m
\end{array}}g(k)=\sum_{k=1}^ng(k)-
\sum_{
k\mid n  \mbox{ or }
k\mid  m
}g(k)=\\
&=&\sum_{k=1}^ng(k)-\sum_{k\mid n}g(k)-\sum_{k\mid m}g(k)+\sum_{k\mid (n,m)}g(k)=a_{ij}.
\end{eqnarray*}
If $i=j=1$
\[
b_{11}=g(1)=a_{11}.
\]
\end{proof}
\noindent
{\bf Particular cases}
\\
{\bf 1.} If $g(n)=\varphi(n)$ then
\[
a_{ij}=
\left\{
\begin{array}{cc}
\displaystyle \sum_{k=1}^n\varphi(k)-i-j+(i,j),&i,j>1\\
\displaystyle \sum_{k=1}^n\varphi(k)-i-j+(i,j)+1,&i=1 \mbox{ or }j=1
\end{array}
\right..
\]
{\bf 2.} If $g(n)=1$ then
\[
a_{ij}=
\left\{
\begin{array}{cc}
n-\tau(i)-\tau(j)+\tau(i,j),&i,j>1\\
n-\tau(i)-\tau(j)+\tau(i,j)+1,&i=1 \mbox{ or }j=1
\end{array}
\right..
\]

\noindent
The following problems remain open:
\begin{problem}
Let $F(n,m)$ be an arithmetical function with two vriables. Determine the structure and the determinant of modified GCD matrices
\linebreak
$A=[a(i,j)]_{n\times n}$, where
\[
a_{ij}=F(i,(i,j)).
\]
\end{problem}
\begin{problem}
Determine the structure and the determinant of modified GCD matrices
$A=[a(i,j)]_{n\times n}$, where
\[
a_{ij}=F(n,i,j,(i,j)).
\]
\end{problem}

%% For references use the following format.
%% For citation use e. g.
%For a single reference use:  \cite{aldo}

%% For more references:

%\medskip
%For more than one references use: \cite{aldo, lothaire, perfect}.

%\medskip
%All references must be cited in the text of your paper. In references, after a red arrow will appear
%the hyperlink backreferences to the pages where they have been cited.
%This can be used for verification if you really cited all references.

%\medskip
%For the usual abbreviation of theoretical journals you can use the address: http://www.ams.org/msnhtml/serials.pdf.

%\medskip
%For acknowledgements instead of footnote, please, use an acknowledgement section at the end of your paper.

%\medskip
%An example of a new reference:  \cite{perfect}.

%% All references must be cited in the text of your paper. In references, after a red arrow will appear
%% the hyperlink backreferences to the pages where they have been cited.

%% For the usual abbreviation of theoretical journals you can use the address: http://www.ams.org/msnhtml/serials.pdf

\subsection*{Acknowledgement}
This research was supported by the grant of Sapientia Foundation, Institute of Scientific Research.

\bigskip
\rightline{\emph{Received: June 12, 2010 }}     %% to be completed by the editor

\end{document}